\documentclass[12pt]{amsart}

\usepackage{amsmath,amsthm,amssymb}
\usepackage[top=30truemm,bottom=30truemm,left=25truemm,right=25truemm]{geometry}
\usepackage[dvips]{graphicx,color,psfrag}

\theoremstyle{plain}
\newtheorem{definition}{Definition}[section]
\newtheorem{thm}[definition]{Theorem}

\newtheorem{lem}[definition]{Lemma}
\newtheorem{cor}[definition]{Corollary}

\newtheorem{rem}[definition]{Remark}

\title{Integer group determinants for ${\rm C}_{2}^{4}$}
\author{Yuka Yamaguchi and Naoya Yamaguchi}
\date{\today}

\begin{document}

\begin{abstract}
We determine all possible values of the integer group determinant of ${\rm C}_{2}^{4}$, 
where ${\rm C}_{2}$ is the cyclic group of order $2$. 
\end{abstract}

\maketitle

\section{Introduction}
For a finite group $G$, 
assigning a variable $x_{g}$ for each $g \in G$, 
the group determinant of $G$ is defined as $\det{\left( x_{g h^{- 1}} \right)}_{g, h \in G}$. 
When the variables $x_{g}$ are all integers, 
the group determinant is called an integer group determinant of $G$. 
For a broad context of group determinants, see their use in the identification of a group \cite{MR1062831}, \cite{MR1123661}, \cite{MR4227663};
the related abelian question of circulants \cite{MR2914452}, \cite{MR3165542}; 
the Lind-Lehmer problem \cite{MR3879399}, \cite{MR3182012}, \cite{MR3165542}; 
and the representation theoretic background \cite{johnson2019group}.

Let $S(G)$ denote the set of all possible values of the integer group determinant of $G$: 
$$
S(G) := \left\{ \det{\left( x_{g h^{- 1}} \right)}_{g, h \in G} \mid x_{g} \in \mathbb{Z} \right\}. 
$$
Let ${\rm C}_{2}$ be the cyclic group of order $2$. 
The complete descriptions of $S \left( {\rm C}_{2}^{2} \right)$ and $S \left( {\rm C}_{2}^{3} \right)$ are obtained in \cite[Theorem~5.3]{MR3879399} and \cite[Theorem~3.1]{MR4056860}, respectively: 
\begin{itemize}
\item $S \left( {\rm C}_{2}^{2} \right) = \left\{ 4 m + 1, \: 2^{4}(2 m + 1), \: 2^{6} m \mid m \in \mathbb{Z} \right\}$; 
\item $S \left( {\rm C}_{2}^{3} \right) =  \left\{ 8 m + 1, \: 2^{8}(4 m + 1), \: 2^{12} m \mid m \in \mathbb{Z} \right\}$. 
\end{itemize}

In this paper,  we determine $S \left( {\rm C}_{2}^{4} \right)$. 
\begin{thm}\label{thm:1.1}
Let $A := \left\{ (8 k - 3) (8 l + 3) \mid k, l \in \mathbb{Z} \right\} \subsetneq \left\{ 8m - 1 \mid m \in \mathbb{Z} \right\}$. 
Then we have 
\begin{align*}
S \left( {\rm C}_{2}^{4} \right) &= \left\{ 16 m + 1, \: 2^{16} (4 m + 1), \: 2^{24} (4 m + 1), \: 2^{24} (8m + 3), \: 2^{24} m', \: 2^{26} m \mid m \in \mathbb{Z}, \: m' \in A \right\}. 
\end{align*}
\end{thm}
Let $\mathbb{Z}_{\rm odd}$ be the set of all odd numbers. 
We remark that all of the odd values in $S \left( {\rm C}_{2}^{4} \right) $ are already known: 
{\it $S \left( {\rm C}_{2}^{n} \right) \cap \mathbb{Z}_{\rm odd} = \left\{ 2^{n} m + 1 \mid m \in \mathbb{Z} \right\}$ holds for any $n$}. 
For example, see \cite[Lemmas~2.1 and 2.2]{MR3182012}. 
Although $S \left( {\rm C}_{2}^{4} \right) \cap \mathbb{Z}_{\rm odd} = \left\{ 16 m + 1 \mid m \in \mathbb{Z} \right\}$ is not new result, 
we give a proof in our method. 

For every group $G$ of order at most $15$, 
$S(G)$ was determined (see \cite{MR4363104}, \cite{MR4056860}). 
For the groups of order $16$, 
the complete descriptions of $S(G)$ were obtained for the dihedral group \cite[Theorem~5.3]{MR3879399} and the cyclic group \cite{Yamaguchi}. 
There are fourteen groups of order $16$ up to isomorphism \cite{MR1510814}, \cite{MR1505615}. 
Theorem~$\ref{thm:1.1}$ determines $S(G)$ for one of the unsolved twelve groups.


\section{Relations with group determinants of subgroups}
For any $g = (\overline{\varepsilon_{0}}, \overline{\varepsilon_{1}}, \ldots, \overline{\varepsilon_{n - 1}}) \in {\rm C}_{2}^{n}$ with $\varepsilon_{i} \in \{ 0, 1 \}$, we denote the variable $x_{g}$ by $x_{j}$, where $j := \varepsilon_{n - 1} \cdot 2^{n - 1} + \varepsilon_{n - 2} \cdot 2^{n - 2} + \cdots + \varepsilon_{0} \cdot 2^{0}$, and let 
$$
D_{n}(x_{0}, x_{1}, \ldots, x_{2^{n} - 1}) := \det{\left( x_{g h^{- 1}} \right)}_{g, h \in {\rm C}_{2}^{n}}. 
$$
From the $H = {\rm C}_{2}^{n - 1}$ and $K = {\rm C}_{2}$ case of \cite[Theorem~1.1]{https://doi.org/10.48550/arxiv.2202.06952}, 
we have the following corollary. 

\begin{cor}\label{cor:2.1}
For any positive integer $n$, 
$$
D_{n + 1}(x_{0}, \ldots, x_{2^{n + 1} - 1}) = D_{n}(x_{0} + x_{2^{n}}, \ldots, x_{2^{n} - 1} + x_{2^{n + 1} - 1}) D_{n}(x_{0} - x_{2^{n}}, \ldots, x_{2^{n} - 1} - x_{2^{n + 1} - 1}). 
$$
\end{cor}
Throughout this paper, 
we assume that $a_{0}, a_{1}, \ldots, a_{15} \in \mathbb{Z}$ and, for any $0 \leq i \leq 3$, 
put 
\begin{align*}
b_{i} &:= (a_{i} + a_{i + 8}) + (a_{i + 4} + a_{i + 12}), & c_{i} &:= (a_{i} + a_{i + 8}) - (a_{i + 4} + a_{i + 12}), \\ 
d_{i} &:= (a_{i} - a_{i + 8}) + (a_{i + 4} - a_{i + 12}), & e_{i} &:= (a_{i} - a_{i + 8}) - (a_{i + 4} - a_{i + 12}). 
\end{align*}

The following relations will be frequently used in this paper: 
\begin{align*}
D_{4}(a_{0}, a_{1}, \ldots, a_{15}) 
&= D_{3}(a_{0} + a_{8}, \ldots, a_{7} + a_{15}) D_{3}(a_{0} - a_{8}, \ldots, a_{7} - a_{15}), \\ 
&= D_{2}(b_{0}, b_{1}, b_{2}, b_{3}) D_{2}(c_{0}, c_{1}, c_{2}, c_{3}) D_{2}(d_{0}, d_{1}, d_{2}, d_{3}) D_{2}(e_{0}, e_{1}, e_{2}, e_{3}). 
\end{align*}

\begin{rem}\label{rem:2.2}
The following hold: 
\begin{enumerate}
\item[$(1)$] $b_{i} \equiv c_{i} \equiv d_{i} \equiv e_{i} \pmod{2} \: \: \text{for} \: \: 0 \leq i \leq 3$; 
\item[$(2)$] $b_{i} + c_{i} + d_{i} + e_{i} \equiv 0 \pmod{4} \: \: \text{for} \: \: 0 \leq i \leq 3$. 
\end{enumerate}
\end{rem}

Noting that $\det{(\alpha_{s t})_{s, t}} \equiv \det{(\beta_{s t})_{s, t}} \pmod{2}$ holds if $\alpha_{s t} \equiv \beta_{s t} \pmod{2}$ for any $s$ and $t$, 
we have the following lemma. 

\begin{lem}\label{lem:2.3}
The following hold: 
\begin{enumerate}
\item[$(1)$] $D_{4}(a_{0}, a_{1}, \ldots, a_{15}) \equiv D_{3}(a_{0} + a_{8}, \ldots, a_{7} + a_{15}) \equiv D_{3}(a_{0} - a_{8}, \ldots, a_{7} - a_{15}) \pmod{2}$; 
\item[$(2)$] $D_{4}(a_{0}, a_{1}, \ldots, a_{15}) \equiv D_{2}(b_{0}, b_{1}, b_{2}, b_{3}) \equiv D_{2}(c_{0}, c_{1}, c_{2}, c_{3}) $\\ 
\hspace{3.36cm}$\equiv D_{2}(d_{0}, d_{1}, d_{2}, d_{3}) \equiv D_{2}(e_{0}, e_{1}, e_{2}, e_{3}) \pmod{2}$. 
\end{enumerate}
\end{lem}

\begin{lem}\label{lem:2.4}
We have 
$D_{4}(a_{0}, a_{1}, \ldots, a_{15}) \in 2 \mathbb{Z} \iff b_{0} + b_{2} \equiv b_{1} + b_{3} \pmod{2}$. 
\end{lem}
\begin{proof}
From Lemma~$\ref{lem:2.3}$~(2) and 
\begin{align*}
D_{2}(b_{0}, b_{1}, b_{2}, b_{3}) 
&= D_{1}(b_{0} + b_{2}, b_{1} + b_{3}) D_{1}(b_{0} - b_{2}, b_{1} - b_{3}) \\
&= \left\{ (b_{0} + b_{2})^{2} - (b_{1} + b_{3})^{2} \right\} \left\{ (b_{0} - b_{2})^{2} - (b_{1} - b_{3})^{2} \right\} \\ 
&\in 
\begin{cases}
2 \mathbb{Z}, & b_{0} + b_{2} \equiv b_{1} + b_{3} \pmod{2}, \\ 
\mathbb{Z}_{\rm odd}, & b_{0} + b_{2} \not\equiv b_{1} + b_{3} \pmod{2}, 
\end{cases}
\end{align*}
the lemma is proved. 
\end{proof}

\section{Group determinant of ${\rm C}_{2}^{2}$}
In this section, we prove five lemmas which give properties of the group determinant of ${\rm C}_{2}^{2}$. 
These lemmas are used in the next section. 

By direct calculation, we have the following lemma. 
\begin{lem}\label{lem:3.1}
The identity $D_{2}(x_{0}, x_{1}, x_{2}, x_{3}) = \sum_{i = 0}^{3} x_{i}^{4} - 2 \sum_{0 \leq i < j \leq 3} x_{i}^{2} x_{j}^{2} + 8 x_{0} x_{1} x_{2} x_{3}$ holds. 
Therefore, $D_{2}(x_{0}, x_{1}, x_{2}, x_{3})$ is a symmetric polynomial in $x_{0}, x_{1}, x_{2}, x_{3}$. 
\end{lem}

\begin{lem}\label{lem:3.2}
For any $k, l, m, n \in \mathbb{Z}$, 
the following hold: 
\begin{enumerate}
\item[$(1)$] $D_{2}(2k, 2l, 2m, 2n + 1) \equiv 8 (k + l + m) + 1 \pmod{16}$; 
\item[$(2)$] $D_{2}(2k, 2 l + 1, 2 m + 1, 2 n + 1) \equiv 8 k - 3 \pmod{16}$. 
\end{enumerate}
\end{lem}
\begin{proof}
We have 
\begin{align*}
&D_{2}(2 k, 2 l, 2 m, 2 n + 1) \\ 
&\hspace{0.5cm} = D_{1}(2k + 2m, 2l + 2n + 1) D_{1}(2k - 2m, 2l - 2n - 1) \\ 
&\hspace{0.5cm} = \left\{ 4 (k + m)^{2} - 4 (l + n)^{2} - 4 (l + n) - 1 \right\} \left\{ 4 (k - m)^{2} - 4 (l - n)^{2} + 4 (l - n) - 1 \right\} \\ 
&\hspace{0.5cm} \equiv - 4 (k + m)^{2} + 4 (l + n)^{2} + 4 (l + n) - 4 (k - m)^{2} + 4 (l - n)^{2} - 4 (l - n) + 1 \\ 
&\hspace{0.5cm} \equiv - 8 k^{2} - 8 m^{2} + 8 l^{2} + 8 n ( n + 1 ) + 1 \\ 
&\hspace{0.5cm} \equiv 8 (k + l + m) + 1 \pmod{16}, \\ 
&D_{2}(2 k, 2 l + 1, 2 m + 1, 2 n + 1) \\ 
&\hspace{0.5cm} = D_{1}(2 k + 2 m + 1, 2 l + 2 n + 2) D_{1}(2 k - 2 m - 1, 2 l - 2 n) \\ 
&\hspace{0.5cm} = \left\{ 4 (k + m)^{2} + 4 (k + m) + 1 - 4 (l + n + 1)^{2} \right\} \left\{ 4 (k - m)^{2} - 4 (k - m) + 1 - 4 (l - n)^{2} \right\} \\ 
&\hspace{0.5cm} \equiv 4 (k + m)^{2} + 4 (k + m) - 4 (l + n + 1)^{2} + 4 (k - m)^{2} - 4 (k - m) - 4 (l - n)^{2} + 1 \\ 
&\hspace{0.5cm} \equiv 8 k^{2} + 8 m (m + 1) - 8 l (l + 1) - 8 n (n + 1) - 3 \\ 
&\hspace{0.5cm} \equiv 8 k - 3 \pmod{16}. 
\end{align*}
\end{proof}

\begin{lem}\label{lem:3.3}
For any $k, l, m, n \in \mathbb{Z}$, 
let $\alpha := D_{2}(2 k, 2 l, 2 m, 2 n)$. 
The following hold: 
\begin{enumerate}
\item[$(1)$] If $k + m \not\equiv l + n, \: k m \equiv l n \pmod{2}$, then $\alpha \in \left\{ 2^{4} (8 a + 1) \mid a \in \mathbb{Z} \right\}$; 
\item[$(2)$] If $k + m \not\equiv l + n, \: k m \not\equiv l n \pmod{2}$, then $\alpha \in \left\{ 2^{4} (8 a - 3) \mid a \in \mathbb{Z} \right\}$; 
\item[$(3)$] If $k, l, m, n$ are even and $k + m \not\equiv l + n \pmod{4}$, then $\alpha \in \left\{ 2^{8} (4 a + 1) \mid a \in \mathbb{Z} \right\}$; 
\item[$(4)$] If $k, l, m, n$ are odd and $k + m \not\equiv l + n \pmod{4}$, then $\alpha \in \left\{ 2^{8} (4 a - 1) \mid a \in \mathbb{Z} \right\}$; 
\item[$(5)$] If $k \equiv l \equiv m \equiv n \pmod{2}$ and $k + m \equiv l + n \pmod{4}$, then $\alpha \in 2^{12} \mathbb{Z}$; 
\item[$(6)$] If exactly two of $k, l, m, n$ are even numbers, then $\alpha \in 2^{10} \mathbb{Z}$. 
\end{enumerate}
\end{lem}
\begin{proof}
First, we prove (1) and (2). 
We have $(k + m)^{2} - (l + n)^{2} = (k - m)^{2} - (l - n)^{2} + 4 (k m - l n)$. 
Therefore, if $k + m \not\equiv l + n \pmod{2}$, then 
\begin{align*}
2^{- 4} \alpha 
&= \left\{ (k + m)^{2} - (l + n)^{2} \right\} \left\{ (k - m)^{2} - (l - n)^{2} \right\} \\ 
&\equiv 
\begin{cases}
1 \pmod{8}, & k m \equiv l n \pmod{2}, \\ 
- 3 \pmod{8}, & k m \not\equiv l n \pmod{2}. 
\end{cases}
\end{align*}
Second, we prove (3). 
If $k, l, m, n$ are even and $k + m \not\equiv l + n \pmod{4}$, 
then there exist $k', l', m', n' \in \mathbb{Z}$ satisfying $k = 2 k'$, $l = 2 l'$, $m = 2 m'$, $n = 2 n'$ and $k' + m' \not\equiv l' + n' \pmod{2}$. 
Therefore, 
$$
2^{- 8} \alpha = \left\{ (k' + m')^{2} - (l' + n')^{2} \right\} \left\{ (k' - m')^{2} - (l' - n')^{2} \right\} \equiv 1 \pmod{4}. 
$$
Third, we prove (4). 
If $k, l, m, n$ are odd and $k + m \not\equiv l + n \pmod{4}$, 
then there exist $k' , l', m', n' \in \mathbb{Z}$ satisfying $k = 2 k' + 1$, $l = 2 l' + 1$, $m = 2 m' + 1$, $n = 2 n' + 1$ and $k' + m' \not\equiv l' + n' \pmod{2}$. 
Therefore, 
\begin{align*}
2^{- 8} \alpha = \left\{ (k' + m' + 1)^{2} - (l' + n' + 1)^{2} \right\} \left\{ (k' - m')^{2} - (l' - n')^{2} \right\} \equiv - 1 \pmod{4}. 
\end{align*}
Fourth, we prove (5). 
If $k \equiv l \equiv m \equiv n \pmod{2}$ and $k + m \equiv l + n \pmod{4}$, 
then $(k + m)^{2} - (l + n)^{2} \equiv (k - m)^{2} - (l - n)^{2} \equiv 0 \pmod{16}$. 
Therefore, 
\begin{align*}
\alpha = 2^{4} \left\{ (k + m)^{2} - (l + n)^{2} \right\} \left\{ (k - m)^{2} - (l - n)^{2} \right\} \in 2^{12} \mathbb{Z}. 
\end{align*}
Finally, we prove (6). 
If $(k, l, m, n) \equiv (0, 0, 1, 1) \pmod{2}$, 
then $(k + m)^{2} - (l + n)^{2} \equiv (k - m)^{2} - (l - n)^{2} \equiv 0 \pmod{8}$. 
Therefore, 
\begin{align*}
\alpha = 2^{4} \left\{ (k + m)^{2} - (l + n)^{2} \right\} \left\{ (k - m)^{2} - (l - n)^{2} \right\} \in 2^{10} \mathbb{Z}. 
\end{align*}
Lemma~$\ref{lem:3.1}$ completes the proof of (6). 
\end{proof}

\begin{lem}\label{lem:3.4}
For any $k, l, m, n \in \mathbb{Z}$, 
let $\alpha := D_{2}(2k, 2l+1, 2m, 2n+1)$. 
The following hold: 
\begin{enumerate}
\item[$(1)$] If $k \equiv m \equiv 0 \pmod{2}$ and {\rm (I)} hold, 
then $\alpha \in \left\{ 2^{6} (8a - 1)(4b - 1) \mid a, b \in \mathbb{Z} \right\}$; 
\item[$(2)$] If $k \equiv m \equiv 0 \pmod{2}$ and {\rm (II)} hold, 
then $\alpha \in \left\{ 2^{6} (8a + 3)(4b + 1) \mid a, b \in \mathbb{Z} \right\}$; 
\item[$(3)$] If $k \equiv m \equiv 1 \pmod{2}$ and {\rm (I)} hold, 
then $\alpha \in \left\{ 2^{6} (8a + 3)(4b - 1) \mid a, b \in \mathbb{Z} \right\}$; 
\item[$(4)$] If $k \equiv m \equiv 1 \pmod{2}$ and {\rm (II)} hold, 
then $\alpha \in \left\{ 2^{6} (8a - 1)(4b +1) \mid a, b \in \mathbb{Z} \right\}$, 
\end{enumerate}
where 
\begin{enumerate}
\item[{\rm (I)}] $k + m \equiv 1 - l - n \equiv 0 \pmod{4}$ or $k - m \equiv 2 - l + n \equiv 0 \pmod{4}$; 
\item[{\rm (II)}] $k + m \equiv 1 - l - n \equiv 2 \pmod{4}$ or $k - m \equiv 2 - l + n \equiv 2 \pmod{4}$. 
\end{enumerate}
\end{lem}
\begin{proof}
Note that 
\begin{align*}
\alpha 
&= D_{1}( 2 k + 2 m, 2 l + 2 n +2 ) D_{1}( 2 k - 2 m, 2 l - 2 n )  \\ 
&= 2^{4} \left\{ (k + m)^{2} - (l + n + 1)^{2} \right\} \left\{ (k - m)^{2} - (l - n)^{2} \right\}. 
\end{align*}
If $k \equiv m \equiv 0 \pmod{2}$, then 
\begin{align*}
(k + m)^{2} - (l + n + 1)^{2} \equiv 
\begin{cases}
- 4 \pmod{16}, & \text{if} \: \: k + m \equiv 1 - l - n \equiv 0 \pmod{4}, \\ 
4 \pmod{16}, & \text{if} \: \: k + m \equiv 1 - l - n \equiv 2 \pmod{4}, \\ 
- 1 \pmod{8}, & \text{if} \: \: k - m \equiv 2 - l + n \equiv 0 \pmod{4}, \\ 
3 \pmod{8}, & \text{if} \: \: k - m \equiv 2 - l + n \equiv 2 \pmod{4}, 
\end{cases} \\ 
(k - m)^{2} - (l - n)^{2} \equiv 
\begin{cases}
- 1 \pmod{8}, & \text{if} \: \: k + m \equiv 1 - l - n \equiv 0 \pmod{4}, \\ 
3 \pmod{8}, & \text{if} \: \: k + m \equiv 1 - l - n \equiv 2 \pmod{4}, \\ 
- 4 \pmod{16}, & \text{if} \: \: k - m \equiv 2 - l + n \equiv 0 \pmod{4}, \\ 
4 \pmod{16}, & \text{if} \: \: k - m \equiv 2 - l + n \equiv 2 \pmod{4}. 
\end{cases}
\end{align*}
If $k \equiv m \equiv 1 \pmod{2}$, 
then 
\begin{align*}
(k + m)^{2} - (l + n + 1)^{2} \equiv 
\begin{cases}
- 4 \pmod{16}, & \text{if} \: \: k + m \equiv 1 - l - n \equiv 0 \pmod{4}, \\ 
4 \pmod{16}, & \text{if} \: \: k + m \equiv 1 - l - n \equiv 2 \pmod{4}, \\ 
3 \pmod{8}, & \text{if} \: \: k - m \equiv 2 - l + n \equiv 0 \pmod{4}, \\ 
- 1 \pmod{8}, & \text{if} \: \: k - m \equiv 2 - l + n \equiv 2 \pmod{4}, 
\end{cases} \\ 
(k - m)^{2} - (l - n)^{2} \equiv 
\begin{cases}
3 \pmod{8}, & \text{if} \: \: k + m \equiv 1 - l - n \equiv 0 \pmod{4}, \\ 
- 1 \pmod{8}, & \text{if} \: \: k + m \equiv 1 - l - n \equiv 2 \pmod{4}, \\ 
- 4 \pmod{16}, & \text{if} \: \: k - m \equiv 2 - l + n \equiv 0 \pmod{4}, \\ 
4 \pmod{16}, & \text{if} \: \: k - m \equiv 2 - l + n \equiv 2 \pmod{4}. 
\end{cases}
\end{align*}
From the above, the lemma is proved. 
\end{proof}

Note that for any integers $a, b$ and $c$, where $b$ is odd, 
it holds that 
$$
a b \equiv c \pmod{8} \iff a \equiv b c \pmod{8} 
$$
since $b^{2} \equiv 1 \pmod{8}$. 
From this, we have the following. 

\begin{rem}\label{rem:3.5}
For any $k, l, m, n \in \mathbb{Z}$, the following hold: 
\begin{enumerate}
\item[$(1)$] $(2 k + 2 l + 1) (2 m + 2 n + 1) \equiv 1 \pmod{8} \iff k - m \equiv - l + n \pmod{4}$; 
\item[$(2)$] $(2 k + 2 l + 1) (2 m + 2 n + 1) \equiv - 1 \pmod{8} \iff k + m \equiv - l - n - 1 \pmod{4}$;  
\item[$(3)$] $(2 k + 2 l + 1) (2 m + 2 n + 1) \equiv 3 \pmod{8} \iff k + m \equiv 1 - l - n \pmod{4}$; 
\item[$(4)$] $(2 k + 2 l + 1) (2 m + 2 n + 1) \equiv - 3 \pmod{8} \iff k - m \equiv 2 - l + n \pmod{4}$.  
\end{enumerate}
\end{rem}

\begin{lem}\label{lem:3.6}
For any $k, l, m, n \in \mathbb{Z}$, 
the following hold: 
\begin{enumerate}
\item[$(1)$] $D_{2}(2 k, 2 l, 2 m, 2 n) \in \begin{cases} 2^{4} \mathbb{Z}_{\rm odd}, & k + m \not\equiv l + n \pmod{2}, \\ 2^{8} \mathbb{Z}_{\rm odd} \cup 2^{10} \mathbb{Z}, & k + m \equiv l + n \pmod{2}; \end{cases}$ 
\item[$(2)$] $D_{2}(2 k +1, 2 l + 1, 2 m + 1, 2 n + 1) \in \begin{cases} 2^{4} \mathbb{Z}_{\rm odd}, & k + m \not\equiv l + n \pmod{2}, \\ 2^{9} \mathbb{Z}, & k + m \equiv l + n \pmod{2}; \end{cases}$ 
\item[$(3)$] $D_{2}(2 k, 2 l + 1, 2 m, 2 n + 1)$ \\ 
\quad $\in \begin{cases} 2^{7} \mathbb{Z}, & k \not\equiv m \pmod{2}, \\ 2^{8} \mathbb{Z}, & k \equiv m \pmod{2} \: \: \text{and} \: \: (2 k + 2 l + 1) (2 m + 2 n + 1) \equiv \pm 1 \pmod{8}, \\ 
2^{6} \mathbb{Z}_{\rm odd}, & k \equiv m \pmod{2} \: \: \text{and} \: \: (2 k + 2 l + 1) (2 m + 2 n + 1) \equiv \pm 3 \pmod{8}. \end{cases}$ 
\end{enumerate}
\end{lem}
\begin{proof}
From Lemma~$\ref{lem:3.3}$, 
it follows that (1) holds. 
Also, we obtain (2) from 
\begin{align*}
D_{2}(2 k +1, 2 l + 1, 2 m + 1, 2 n + 1) 
&= D_{1}(2 k + 2 m + 2, 2 l + 2 n + 2) D_{1}(2 k - 2 m, 2 l - 2 n) \\ 
&= 2^{4} \left\{ (k + m + 1)^{2} - (l + n + 1)^{2} \right\} \left\{ (k - m)^{2} - (l - n)^{2} \right\}. 
\end{align*} 
Below, we prove (3). 
Let $\alpha := D_{2}(2 k, 2 l + 1, 2 m, 2 n + 1)$. 
Note that 
\begin{align*}
\alpha 
&= D_{1}(2 k + 2 m, 2 l + 2 n + 2) D_{1}(2 k - 2 m, 2 l - 2 n) \\ 
&= 2^{4} \left\{ (k + m)^{2} - (l + n + 1)^{2} \right\} \left\{ (k - m)^{2} - (l - n)^{2} \right\}. 
\end{align*}
If $k \not\equiv m \pmod{2}$, then $\alpha \in 2^{7} \mathbb{Z}$. 
If $k \equiv m \pmod{2}$ and $(2 k + 2 l + 1) (2 m + 2 n + 1) \equiv 1 \pmod{8}$, 
then $(k - m)^{2} - (l - n)^{2} \in 2^{4} \mathbb{Z}$ from Remark~$\ref{rem:3.5}$~$(1)$. 
Thus, $\alpha \in 2^{8} \mathbb{Z}$. 
If $k \equiv m \pmod{2}$ and $(2 k + 2 l + 1) (2 m + 2 n + 1) \equiv - 1 \pmod{8}$, 
then $(k + m)^{2} - (l + n + 1)^{2} \in 2^{4} \mathbb{Z}$ from Remark~$\ref{rem:3.5}$~$(2)$. 
Thus, $\alpha \in 2^{8} \mathbb{Z}$. 
The cases of $k \equiv m \pmod{2}$ and $(2 k + 2 l + 1) (2 m + 2 n + 1) \equiv \pm 3 \pmod{8}$ are proved from Lemma~$\ref{lem:3.4}$ and Remark~$\ref{rem:3.5}$~$(3)$ and $(4)$. 
\end{proof}

\section{Impossible values}\label{Section3}

In this section, 
we consider impossible values. 

\begin{lem}\label{lem:4.1}
We have $S \left( {\rm C}_{2}^{4} \right) \cap \mathbb{Z}_{\rm odd} \subset \left\{ 16 m + 1 \mid m \in \mathbb{Z} \right\}$. 
\end{lem}
\begin{proof}
Let $D_{4}(a_{0}, \ldots, a_{15}) = D_{2}(b_{0}, b_{1}, b_{2}, b_{3}) D_{2}(c_{0}, c_{1}, c_{2}, c_{3}) D_{2}(d_{0}, d_{1}, d_{2}, d_{3}) D_{2}(e_{0}, e_{1}, e_{2}, e_{3})$ be an odd number. 
Then, $b_{0} + b_{2} \not\equiv b_{1} + b_{3} \pmod{2}$ holds from Lemma~$\ref{lem:2.4}$. 
We divide the proof into the following cases: 
\begin{enumerate}
\item[(i)] Exactly three of $b_{0}, b_{1}, b_{2}, b_{3}$ are even; 
\item[(ii)] Exactly one of $b_{0}, b_{1}, b_{2}, b_{3}$ is even. 
\end{enumerate}
First, we consider the case of (i). 
If $( b_{0}, b_{1}, b_{2}, b_{3} ) \equiv (0, 0, 0, 1) \pmod{2}$, 
then there exist $k_{i}, l_{i}, m_{i} \in \mathbb{Z}$ satisfying $(b_{0}, b_{1}, b_{2}) = (2 k_{0}, 2 l_{0}, 2 m_{0})$, 
$(c_{0}, c_{1}, c_{2}) = (2 k_{1}, 2 l_{1}, 2 m_{1})$, 
$(d_{0}, d_{1}, d_{2}) = (2 k_{2}, 2 l_{2}, 2 m_{2})$, 
$(e_{0}, e_{1}, e_{2}) = (2 k_{3}, 2 l_{3}, 2 m_{3})$ and $\sum_{i = 0}^{3} k_{i} \equiv \sum_{i = 0}^{3} l_{i} \equiv \sum_{i = 0}^{3} m_{i} \equiv 0 \pmod{2}$ from Remark~$\ref{rem:2.2}$. 
Therefore, from Lemma~$\ref{lem:3.2}$~$(1)$, we have 
\begin{align*}
D_{4}(a_{0}, a_{1}, \ldots, a_{15}) &\equiv 
\prod_{i = 0}^{3} \left\{ 8 (k_{i} + l_{i} + m_{i}) + 1 \right\} \equiv 8 \sum_{i = 0}^{3} (k_{i} + l_{i} + m_{i}) + 1 \equiv 1 \pmod{16}. 
\end{align*}
Lemma~$\ref{lem:3.1}$ completes the proof for the case (i). 
Next, we consider the case of (ii). 
If $( b_{0}, b_{1}, b_{2}, b_{3} ) \equiv (0, 1, 1, 1) \pmod{2}$, 
then there exist $k'_{i} \in \mathbb{Z}$ satisfying $b_{0} = 2 k'_{0}$, 
$c_{0} = 2 k'_{1}$, $d_{0} = 2 k'_{2}$, $e_{0} = 2 k' _{3}$ and $\sum_{i = 0}^{3} k'_{i} \equiv 0 \pmod{2}$ from Remark~$\ref{rem:2.2}$. 
Therefore, from Lemma~$\ref{lem:3.2}$~$(2)$, we have 
\begin{align*}
D_{4}(a_{0}, a_{1}, \ldots, a_{15}) \equiv \prod_{i = 0}^{3} (8 k'_{i} - 3) \equiv 8 \sum_{i = 0}^{3} k'_{i} + 1 \equiv 1 \pmod{16}. 
\end{align*}
Lemma~$\ref{lem:3.1}$ completes the proof for the case (ii). 
\end{proof}

\begin{lem}\label{lem:4.2}
We have $S \left( {\rm C}_{2}^{4} \right) \cap 2 \mathbb{Z} \subset 2^{16} \mathbb{Z}_{\rm odd} \cup 2^{24} \mathbb{Z}_{\rm odd} \cup 2^{26} \mathbb{Z}$. 
\end{lem}
\begin{proof}
Let $D_{4}(a_{0}, \ldots, a_{15}) = D_{2}(b_{0}, b_{1}, b_{2}, b_{3}) D_{2}(c_{0}, c_{1}, c_{2}, c_{3}) D_{2}(d_{0}, d_{1}, d_{2}, d_{3}) D_{2}(e_{0}, e_{1}, e_{2}, e_{3})$ be an even number. 
Then, $b_{0} + b_{2} \equiv b_{1} + b_{3} \pmod{2}$ holds from Lemma~$\ref{lem:2.4}$. 
We prove the following: 
\begin{enumerate}
\item[(i)] If $b_{0}, b_{1}, b_{2}, b_{3}$ are even, then $D_{4}(a_{0}, \ldots, a_{15}) \in 2^{16} \mathbb{Z}_{\rm odd} \cup 2^{24} \mathbb{Z}_{\rm odd} \cup 2^{26} \mathbb{Z}$; 
\item[(ii)] If $b_{0}, b_{1}, b_{2}, b_{3}$ are odd, then $D_{4}(a_{0}, \ldots, a_{15}) \in 2^{16} \mathbb{Z}_{\rm odd} \cup 2^{26} \mathbb{Z}$; 
\item[(iii)] If exactly two of $b_{0}, b_{1}, b_{2}, b_{3}$ are even, then $D_{4}(a_{0}, \ldots, a_{15}) \in 2^{24} \mathbb{Z}_{\rm odd} \cup 2^{26} \mathbb{Z}$. 
\end{enumerate}
First, we prove (i). 
If $b_{0}, b_{1}, b_{2}, b_{3}$ are even, 
then there exist $k_{i}, l_{i}, m_{i}, n_{i} \in \mathbb{Z}$ satisfying 
\begin{align*}
(b_{0}, b_{1}, b_{2}, b_{3}) &= (2 k_{0}, 2 l_{0}, 2 m_{0}, 2 n_{0}), &
(c_{0}, c_{1}, c_{2}, c_{3}) = (2 k_{1}, 2 l_{1}, 2 m_{1}, 2 n_{1}), \\ 
(d_{0}, d_{1}, d_{2}, d_{3}) &= (2 k_{2}, 2 l_{2}, 2 m_{2}, 2 n_{2}), &
(e_{0}, e_{1}, e_{2}, e_{3}) = (2 k_{3}, 2 l_{3}, 2 m_{3}, 2 n_{3})
\end{align*}
and $\sum_{i = 0}^{3} k_{i} \equiv \sum_{i = 0}^{3} l_{i} \equiv \sum_{i = 0}^{3} m_{i} \equiv \sum_{i = 0}^{3} n_{i} \equiv 0 \pmod{2}$ from Remark~$\ref{rem:2.2}$. 
Let $N$ denote the cardinal number of the set $\left\{ i \mid 0 \leq i \leq 3, \: \: k_{i} + m_{i} \equiv l_{i} + n_{i} \pmod{2} \right\}$. 
Then, $N \in \{ 0, 2, 4 \}$ holds from $\sum_{i = 0}^{3} (k_{i} + m_{i} - l_{i} - n_{i}) \equiv 0 \pmod{2}$. 
Therefore, 
from Lemma~$\ref{lem:3.6}$~$(1)$, 
we obtain (i). 
Also, in the same way, 
we can prove (ii) by using Lemma~$\ref{lem:3.6}$~(2). 
Finally, we prove (iii). 
If $( b_{0}, b_{1}, b_{2}, b_{3} ) \equiv (0, 1, 0, 1) \pmod{2}$, 
then there exist $k'_{i}, m'_{i} \in \mathbb{Z}$ satisfying $(b_{0}, b_{2}) = (2 k'_{0}, 2 m'_{0})$, 
$(c_{0}, c_{2}) = (2 k'_{1}, 2 m'_{1})$, 
$(d_{0}, d_{2}) = (2 k'_{2}, 2 m'_{2})$, 
$(e_{0}, e_{2}) = (2 k'_{3}, 2 m'_{3})$ and $\sum_{i = 0}^{3} k'_{i} \equiv \sum_{i = 0}^{3} m'_{i} \equiv 0 \pmod{2}$ from Remark~$\ref{rem:2.2}$. 
Let $N'$ denote the cardinal number of the set $\left\{ i \mid 0 \leq i \leq 3, \: \: k'_{i} \equiv m'_{i} \pmod{2} \right\}$. 
Then, $N' \in \{ 0, 2, 4 \}$ holds from $\sum_{i = 0}^{3} (k'_{i} - m'_{i}) \equiv 0 \pmod{2}$. 
Therefore, 
from Lemma~$\ref{lem:3.6}$~$(3)$, 
we have $D_{4}(a_{0}, \ldots, a_{15}) \in 2^{24} \mathbb{Z}_{\rm odd} \cup 2^{26} \mathbb{Z}$. 
Lemma~$\ref{lem:3.1}$ completes the proof of (iii). 
\end{proof}

\begin{lem}\label{lem:}
We have $S \left( {\rm C}_{2}^{4} \right) \cap 2^{16} \mathbb{Z}_{\rm odd} \subset \left\{ 2^{16} (4 m + 1) \mid m \in \mathbb{Z} \right\}$. 
\end{lem}
\begin{proof}
Note that we have $S \left( {\rm C}_{2}^{3} \right) \cap 2 \mathbb{Z} = \left\{ 2^{8} (4 m + 1), 2^{12} m \mid m \in \mathbb{Z} \right\}$ from the description in the introduction. 
Let $D_{4}(a_{0}, \ldots, a_{15}) = D_{3}(a_{0} + a_{8}, \ldots, a_{7} + a_{15}) D_{3}(a_{0} - a_{8}, \ldots, a_{7} - a_{15}) \in 2^{16} \mathbb{Z}_{\rm odd}$. 
Then, from Lemma~$\ref{lem:2.3}$~$(1)$, 
there exist $k, l \in \mathbb{Z}$ satisfying 
$D_{3}(a_{0} + a_{8}, \ldots, a_{7} + a_{15}) = 2^{8} (4 k + 1)$ and $D_{3}(a_{0} - a_{8}, \ldots, a_{7} - a_{15}) = 2^{8} (4 l + 1)$. 
Therefore, we have $D_{4}(a_{0}, \ldots, a_{15}) \in \left\{ 2^{16} (4 m + 1) \mid m \in \mathbb{Z} \right\}$. 
\end{proof}

Let $A := \left\{ (8 k - 3) (8 l + 3) \mid k, l \in \mathbb{Z} \right\}$. 

\begin{lem}\label{lem:4.4}
If $m \equiv - 1 \pmod{8}$ and $m \not\in A$, 
then $2^{24} m \not\in S \left( {\rm C}_{2}^{4} \right)$. 
\end{lem}

To prove Lemma~$\ref{lem:4.4}$, we use the following two lemmas. 

\begin{lem}\label{lem:4.5}
If $m \equiv - 1 \pmod{8}$ and $m \not\in A$, 
then $D_{4}(a_{0}, a_{1}, \ldots, a_{15}) \neq 2^{24} m$ for any $a_{0}, a_{1}, \ldots, a_{15} \in \mathbb{Z}$ with $b_{0} \equiv b_{1} \equiv b_{2} \equiv b_{3} \equiv 0 \pmod{2}$. 

\end{lem}
\begin{proof}
We prove by contradiction. 
Assume that there exist $a_{0}, a_{1}, \ldots, a_{15} \in \mathbb{Z}$ with $b_{0} \equiv b_{1} \equiv b_{2} \equiv b_{3} \equiv 0 \pmod{2}$ satisfying $D_{4}(a_{0}, a_{1}, \ldots, a_{15}) = 2^{24} m$. 
Then there exist $k_{i}, l_{i}, m_{i}, n_{i} \in \mathbb{Z}$ satisfying 
\begin{align*}
D_{2}^{(0)} &:= D_{2}(2k_{0}, 2l_{0}, 2m_{0}, 2n_{0}) = D_{2}(b_{0}, b_{1}, b_{2}, b_{3}), \\ 
D_{2}^{(1)} &:= D_{2}(2k_{1}, 2l_{1}, 2m_{1}, 2n_{1}) = D_{2}(c_{0}, c_{1}, c_{2}, c_{3}), \\ 
D_{2}^{(2)} &:= D_{2}(2k_{2}, 2l_{2}, 2m_{2}, 2n_{2}) = D_{2}(d_{0}, d_{1}, d_{2}, d_{3}), \\ 
D_{2}^{(3)} &:= D_{2}(2k_{3}, 2l_{3}, 2m_{3}, 2n_{3}) = D_{2}(e_{0}, e_{1}, e_{2}, e_{3}) 
\end{align*}
and $\sum_{i = 0}^{3} k_{i} \equiv \sum_{i = 0}^{3} l_{i} \equiv \sum_{i = 0}^{3} m_{i} \equiv \sum_{i = 0}^{3} n_{i} \equiv 0 \pmod{2}$ from Remark~$\ref{rem:2.2}$. 
Note that $\sum_{i = 0}^{3} ( k_{i} + m_{i} - l_{i} - n_{i} ) \equiv 0 \pmod{2}$. 
Since $D_{2}^{(0)} D_{2}^{(1)} D_{2}^{(2)} D_{2}^{(3)} \in 2^{24} \mathbb{Z}_{\rm odd}$, one of the following holds: 
(i) Three of $D_{2}^{(i)}$ are of type (1) or (2) in Lemma~$\ref{lem:3.3}$ and the other one is of type (5) in Lemma~$\ref{lem:3.3}$; 
(ii) Two of $D_{2}^{(i)}$ are of type (1) or (2) in Lemma~$\ref{lem:3.3}$ and the others are of type (3) or (4) in Lemma~$\ref{lem:3.3}$. 
In the case of (i), we have $\sum_{i = 0}^{3} ( k_{i} + m_{i} - l_{i} - n_{i} ) \not\equiv 0 \pmod{2}$. 
This is a contradiction. 
In the case of (ii), since $m \equiv - 1 \pmod{8}$ and $m \not\in A$, we have 
\begin{align*}
D_{2}^{(\alpha)}, D_{2}^{(\beta)} \in \left\{ 2^{4} (8 a + 1) \mid a \in \mathbb{Z} \right\}, 
D_{2}^{(\gamma)} \in \left\{ 2^{8} (4 a + 1) \mid a \in \mathbb{Z} \right\}, 
D_{2}^{(\delta)} \in \left\{ 2^{8} (4 a - 1) \mid a \in \mathbb{Z} \right\}, 
\end{align*}
where $\left\{ \alpha, \beta, \gamma, \delta \right\} = \left\{ 0, 1, 2, 3 \right\}$. 
This implies that 
\begin{align*}
(k_{\alpha}, l_{\alpha}, m_{\alpha}, n_{\alpha}) &\equiv (0, 0, 0, 1), \: (0, 0, 1, 0), \: (0, 1, 0, 0) \: \text{or} \: (1, 0, 0, 0) \pmod{2}, \\ 
(k_{\beta}, l_{\beta}, m_{\beta}, n_{\beta}) &\equiv (0, 0, 0, 1), \: (0, 0, 1, 0), \: (0, 1, 0, 0) \: \text{or} \: (1, 0, 0, 0) \pmod{2}, \\ 
(k_{\gamma}, l_{\gamma}, m_{\gamma}, n_{\gamma}) &\equiv (0, 0, 0, 0) \pmod{2}, \\ 
(k_{\delta}, l_{\delta}, m_{\delta}, n_{\delta}) &\equiv (1, 1, 1, 1) \pmod{2}. 
\end{align*}
Therefore, at least two of $\sum_{i = 0}^{3} k_{i}$, $\sum_{i = 0}^{3} l_{i}$, $\sum_{i = 0}^{3} m_{i}$, $\sum_{i = 0}^{3} n_{i}$ are odd. 
This is a contradiction. 
\end{proof}

\begin{lem}\label{lem:4.6}
If $m \equiv - 1 \pmod{8}$ and $m \not\in A$, then $D_{4}(a_{0}, a_{1}, \ldots, a_{15}) \neq 2^{24} m$ for any $a_{0}, a_{1}, \ldots, a_{15} \in \mathbb{Z}$, 
where exactly two of $b_{0}, b_{1}, b_{2}, b_{3}$ are even. 
\end{lem}
\begin{proof}
We prove by contradiction. 
Assume that there exist $a_{0}, a_{1}, \ldots, a_{15} \in \mathbb{Z}$, 
where exactly two of $b_{0}, b_{1}, b_{2}, b_{3}$ are even, 
satisfying $D_{4}(a_{0}, a_{1}, \ldots, a_{15}) = 2^{24} m$. 
If $(b_{0}, b_{1}, b_{2}, b_{3}) \equiv (0, 1, 0, 1) \pmod{2}$, 
then there exist $k_{i}, l_{i}, m_{i}, n_{i} \in \mathbb{Z}$ satisfying 
\begin{align*}
D_{2}^{(0)} &:= D_{2}(2k_{0}, 2l_{0} + 1, 2m_{0}, 2n_{0} + 1) = D_{2}(b_{0}, b_{1}, b_{2}, b_{3}), \\ 
D_{2}^{(1)} &:= D_{2}(2k_{1}, 2l_{1} + 1, 2m_{1}, 2n_{1} + 1) = D_{2}(c_{0}, c_{1}, c_{2}, c_{3}), \\ 
D_{2}^{(2)} &:= D_{2}(2k_{2}, 2l_{2} + 1, 2m_{2}, 2n_{2} + 1) = D_{2}(d_{0}, d_{1}, d_{2}, d_{3}), \\ 
D_{2}^{(3)} &:= D_{2}(2k_{3}, 2l_{3} + 1, 2m_{3}, 2n_{3} + 1) = D_{2}(e_{0}, e_{1}, e_{2}, e_{3}) 
\end{align*}
and $\sum_{i = 0}^{3} k_{i} \equiv \sum_{i = 0}^{3} l_{i} \equiv \sum_{i = 0}^{3} m_{i} \equiv \sum_{i = 0}^{3} n_{i} \equiv 0 \pmod{2}$ from Remark~$\ref{rem:2.2}$. 
Since $D_{2}^{(0)} D_{2}^{(1)} D_{2}^{(2)} D_{2}^{(3)} \in 2^{24} \mathbb{Z}_{\rm odd}$,  
it follows from Lemma~$\ref{lem:3.6}$~$(3)$ that $k_{i} \equiv m_{i} \pmod{2}$ and $( 2k_{i} + 2l_{i} + 1) ( 2 m_{i} + 2 n_{i} + 1 ) \equiv \pm 3 \pmod{8}$ for any $0 \leq i \leq 3$. 
Then, from Remark~$\ref{rem:3.5}$~$(3)$ and $(4)$, we find that every $D_{2}^{(i)}$ is of type $(1)$ or $(4)$ in Lemma~$\ref{lem:3.4}$ since $m \not\in A$. 
In addition, since $m \equiv - 1 \pmod{8}$, 
exactly one or three of $D_{2}^{(i)}$ are of type $(1)$ in Lemma~$\ref{lem:3.4}$.  
Then we have $\sum_{i = 0}^{3} k_{i} \equiv 1 \pmod{2}$. This is a contradiction. 
Lemma~$\ref{lem:3.1}$ completes the proof. 
\end{proof}

\begin{proof}[Proof of Lemma~$\ref{lem:4.4}$]
Let $m \equiv - 1 \pmod{8}$ and $m \not\in A$. 
From Lemma~$\ref{lem:2.4}$, 
it is sufficient to prove that $D_{4}(a_{0}, \ldots, a_{15}) \neq 2^{24} m$ for any $a_{0}, \ldots, a_{15} \in \mathbb{Z}$ with $b_{0} + b_{2} \equiv b_{1} + b_{3} \pmod{2}$. 
 If $b_{0}, b_{1}, b_{2}, b_{3}$ are odd, 
 then as mentioned in the proof of Lemma~$\ref{lem:4.2}$, 
 it holds that $D_{4}(a_{0}, \ldots, a_{15}) \in 2^{16} \mathbb{Z}_{\rm odd} \cup 2^{26} \mathbb{Z}$. 
Thus, $D_{4}(a_{0}, \ldots, a_{15}) \neq 2^{24} m$. 
If all or exactly two of $b_{0}, b_{1}, b_{2}, b_{3}$ are even, then we have 
$D_{4}(a_{0}, \ldots, a_{15}) \neq 2^{24} m$ from Lemmas~$\ref{lem:4.5}$ and $\ref{lem:4.6}$. 
\end{proof}

\section{Possible values}
In this section, 
we determine all possible values. 
Lemmas~$\ref{lem:4.1}$--$\ref{lem:4.4}$ imply that $S \left( {\rm C}_{2}^{4} \right)$ does not include every integer that is not mentioned in the following lemma. 

\begin{lem}\label{lem:5.1}
For any $m, n \in \mathbb{Z}$, 
the following hold: 
\begin{enumerate}
\item[$(1)$] $16 m + 1 \in S \left( {\rm C}_{2}^{4} \right)$; 
\item[$(2)$] $2^{16} (4 m + 1) \in S \left( {\rm C}_{2}^{4} \right)$; 
\item[$(3)$] $2^{24} (4 m + 1) \in S \left( {\rm C}_{2}^{4} \right)$; 
\item[$(4)$] $2^{24} (4 m + 1) (8 n + 3) \in S \left( {\rm C}_{2}^{4} \right)$; 
\item[$(5)$] $2^{26} m \in S \left( {\rm C}_{2}^{4} \right)$. 
\end{enumerate}
\end{lem}
\begin{proof}
We obtain (1) from 
\begin{align*}
D_{4}(m + 1, m, m, \ldots, m) 
&= D_{3}(2m + 1, 2m, 2m, \ldots, 2m) D_{3}(1, 0, 0, \ldots, 0) \\ 
&= D_{2}(4m + 1, 4m, 4m, 4m) D_{2}(1, 0, 0, 0)^{3} \\ 
&= D_{1}(8m + 1, 8m) D_{1}(1, 0)^{7} \\ 
&= (8m + 1)^{2} - (8m)^{2} \\ 
&= 16 m + 1. 
\end{align*}
We obtain (2) from 
\begin{align*}
D_{4}(k + 2, k, k, \ldots, k) 
&= D_{3}(2k + 2, 2k, 2k, \ldots, 2k) D_{3}(2, 0, 0, \ldots, 0) \\ 
&= D_{2}(4 k + 2, 4 k, 4 k, 4 k) D_{2}(2, 0, 0, 0)^{3} \\ 
&= D_{1}(8k + 2, 8k) D_{1}(2, 0)^{7} \\ 
&= \left\{ (8 k + 2)^{2} - (8 k)^{2} \right\} \cdot 4^{7} \\ 
&=2^{16} (8 k + 1) 
\end{align*}
and 
\begin{align*}
&D_{4}(1 - k, k, 1 - k, k, - k, k, 1 - k, k - 1, 1 - k, k, - k, k - 1, - k, k, - k, k) \\ 
&\quad = D_{3}(2 - 2k, 2k, 1 - 2k, 2k - 1, - 2k, 2k, 1- 2k, 2k - 1) D_{3}(0, 0, 1, 1, 0, 0, 1, - 1) \\ 
&\quad = D_{2}(2 - 4k, 4k, 2 - 4k, 4k - 2) D_{2}(2, 0, 0, 0) D_{2}(0, 0, 2, 0) D_{2}(0, 0, 0, 2) \\ 
&\quad = D_{1}(4 - 8k, 8k - 2) D_{1}(0, 2)^{2} D_{1}(2, 0)^{3} D_{1}(- 2, 0) D_{1}(0, - 2) \\ 
&\quad = \left\{ (4 - 8k)^{2} - (8k - 2)^{2} \right\} (- 4)^{3} \cdot 4^{4} \\ 
&\quad = 2^{16} (8k - 3). 
\end{align*}
We obtain (3) from 
\begin{align*}
D_{4}(m + 3, m + 1, m, m, \ldots, m) 
&= D_{3}(2m + 3, 2m + 1, 2m, 2m, \ldots, 2m) D_{3}(3, 1, 0, 0, \ldots, 0) \\ 
&= D_{2}(4 m + 3, 4 m + 1, 4m, 4m) D_{2}(3, 1, 0, 0)^{3} \\ 
&= D_{1}(8m + 3, 8m + 1) D_{1}(3, 1)^{7} \\ 
&= \left\{ (8m + 3)^{2} - (8m + 1)^{2} \right\} \cdot 8^{7} \\ 
&= 2^{24} (4m + 1). 
\end{align*}
We obtain (4) by substituting 
\begin{align*}
a_{0} = m - n + 2, \: \: 
a_{1} = - (m + n + 1), \: \: 
a_{2} = m + n + 1, \: \: 
a_{3} = n - m, \\ 
a_{4} = m - n - 1, \: \: 
a_{5} = - (m + n), \: \: 
a_{6} = m + n + 1, \: \: 
a_{7} = n - m, \\ 
a_{8} = m - n - 1, \: \: 
a_{9} = - (m + n), \: \: 
a_{10} = m + n + 1, \: \: 
a_{11} = n - m, \\ 
a_{12} = m - n - 1, \: \: 
a_{13} = - (m + n), \: \: 
a_{14} = m + n + 1, \: \: 
a_{15} = n - m. 
\end{align*}
In fact, 
\begin{align*}
D_{4}(a_{0}, \ldots, a_{15}) 
&= D_{2}(b_{0}, b_{1}, b_{2}, b_{3}) D_{2}(c_{0}, c_{1}, c_{2}, c_{3}) D_{2}(d_{0}, d_{1}, d_{2}, d_{3}) D_{2}(e_{0}, e_{1}, e_{2}, e_{3}) \\ 
&= D_{2}( 4 m - 4 n - 1, - 4 m - 4 n - 1, 4 m + 4 n + 4, 4 n - 4 m ) D_{2}(3, - 1, 0, 0)^{3} \\ 
&= D_{1}(8m + 3, - 8m - 1) D_{1}(- 8 n - 5, - 8 n - 1) D_{1}(3, - 1)^{6} \\ 
&= \left\{ (8m + 3)^{2} - (8m + 1)^{2} \right\} \left\{ (8 n + 5)^{2} - (8 n + 1)^{2} \right\} \cdot 8^{6} \\ 
&= 2^{24} (4m + 1) (8 n + 3). 
\end{align*}
We obtain (5) from 
\begin{align*}
&D_{4}(k + 1, 1 - k, k + 1, - k, k - 1, - k, k + 1, - k, k, - k, k + 1, - k, k - 2, 1 - k, k + 1, - k) \\ 
&\quad = D_{3}(2 k + 1, 1 - 2 k, 2 k + 2, - 2 k, 2 k - 3, 1 - 2 k, 2 k + 2, - 2 k) D_{3}(1, 1, 0, 0, 1, - 1, 0, 0) \\ 
&\quad = D_{2}(4 k - 2, 2 - 4 k, 4 k + 4, - 4 k) D_{2}(4, 0, 0, 0) D_{2}(2, 0, 0, 0) D_{2}(0, 2, 0, 0) \\ 
&\quad = D_{1}(8 k + 2, 2 - 8 k) D_{1}(- 6, 2) D_{1}(4, 0)^{2} D_{1}(2, 0)^{2} D_{1}(0, 2)^{2} \\ 
&\quad = \left\{ (8 k + 2)^{2} - (2 - 8 k)^{2} \right\} \cdot 32 \cdot 16^{2} \cdot 4^{2} \cdot (- 4)^{2} \\ 
&\quad = 2^{26} (2 k) 
\end{align*}
and 
\begin{align*}
&D_{4}(k - 1, k + 1, k + 1, k + 2, k + 1, k + 2, k + 2, k, k, k, k, k, k, k, k, k) \\ 
&\quad = D_{3}(2 k - 1, 2 k + 1, 2 k + 1, 2 k + 2, 2 k + 1, 2 k + 2, 2 k + 2, 2 k) D_{3}(- 1, 1, 1, 2, 1, 2, 2, 0) \\ 
&\quad = D_{2}(4 k, 4 k + 3, 4 k + 3, 4 k + 2) D_{2}(- 2, - 1, - 1, 2)^{2} D_{2}(0, 3, 3, 2) \\ 
&\quad = D_{1}(8 k + 3, 8 k + 5) D_{1}(- 3, 1)^{4} D_{1}(- 1, - 3)^{2} D_{1}(3, 5) \\ 
&\quad = \left\{ (8 k + 3)^{2} - (8 k + 5)^{2} \right\} \cdot 8^{4} \cdot (- 8)^{2} \cdot (- 16) \\ 
&\quad = 2^{26} (2 k + 1). 
\end{align*}
\end{proof}

Remark that Lemma~$\ref{lem:5.1}$~$(4)$ implies that $2^{24} (8 m + 3), 2^{24} m' \in S \left( {\rm C}_{2}^{4} \right)$ for any $m \in \mathbb{Z}$ and $m' \in A$. 
From Lemmas~$\ref{lem:4.1}$--$\ref{lem:4.4}$ and $\ref{lem:5.1}$, Theorem~$\ref{thm:1.1}$ is proved.

\clearpage

\bibliography{reference}
\bibliographystyle{plain}

\medskip
\begin{flushleft}
Yuka Yamaguchi\\
Faculty of Education \\ 
University of Miyazaki \\
1-1 Gakuen Kibanadai-nishi\\ 
Miyazaki 889-2192 \\
JAPAN\\
y-yamaguchi@cc.miyazaki-u.ac.jp
\end{flushleft}

\medskip
\begin{flushleft}
Naoya Yamaguchi\\
Faculty of Education \\ 
University of Miyazaki \\
1-1 Gakuen Kibanadai-nishi\\ 
Miyazaki 889-2192 \\
JAPAN\\
n-yamaguchi@cc.miyazaki-u.ac.jp
\end{flushleft}

\end{document}